\documentclass[12pt]{amsart}
\usepackage{amssymb, amstext, amscd, amsmath}
\makeatletter
\def\@cite#1#2{{\m@th\upshape\bfseries%
[{#1\if@tempswa{\m@th\upshape\mdseries, #2}\fi}]}}
\makeatother

\theoremstyle{plain}
\newtheorem{thm}{Theorem}[section]
\newtheorem{cor}[thm]{Corollary}
\newtheorem{prop}[thm]{Proposition}
\newtheorem{lem}[thm]{Lemma}

\theoremstyle{definition}
\newtheorem{rem}[thm]{Remark}

\newtheorem{defn}[thm]{Definition}

\newcommand{\bbD}{{\mathbb{D}}}

\newcommand{\bbN}{{\mathbb{N}}}

\newcommand{\bbT}{{\mathbb{T}}}
\newcommand{\bbZ}{{\mathbb{Z}}}

  \newcommand{\A}{{\mathcal{A}}}
  \newcommand{\B}{{\mathcal{B}}}
  \newcommand{\C}{{\mathcal{C}}}

  \newcommand{\F}{{\mathcal{F}}}
  \newcommand{\G}{{\mathcal{G}}}
\renewcommand{\H}{{\mathcal{H}}}
  
  \newcommand{\J}{{\mathcal{J}}}
  \newcommand{\K}{{\mathcal{K}}}
\renewcommand{\L}{{\mathcal{L}}}
  \newcommand{\M}{{\mathcal{M}}}
  
\renewcommand{\O}{{\mathcal{O}}}

  \newcommand{\T}{{\mathcal{T}}}

  \newcommand{\X}{{\mathcal{X}}}

\renewcommand{\phi}{\varphi}
\newcommand{\upchi}{{\raise.35ex\hbox{\ensuremath{\chi}}}}

\def\ga{\alpha}

\newcommand{\nor}[1]{\left\Vert #1\right\Vert}
\newcommand{\fA}{{\mathfrak{A}}}
\newcommand{\fB}{{\mathfrak{B}}}

\newcommand{\fH}{{\mathfrak{H}}}

\newcommand{\fV}{{\mathfrak{V}}}

\newcommand{\rC}{{\mathrm{C}}}

\newcommand{\foral}{\text{ for all }}

\newcommand{\qfor}{\quad\text{for}\quad}
\newcommand{\qforal}{\quad\text{for all}\,}

\newcommand{\Alg}{\operatorname{Alg}}
\newcommand{\alg}{\operatorname{alg}}

\newcommand{\id}{{\operatorname{id}}}
\newcommand{\is}{{\operatorname{is}}}

\DeclareMathOperator*{\sotlim}{\textsc{sot}--lim}

\newcommand{\ca}{\mathrm{C}^*}
\newcommand{\cenv}{\mathrm{C}^*_{\text{env}}}

\newcommand{\ol}{\overline}
\newcommand{\tpi}{\widetilde{\pi}}
\newcommand{\hpi}{\widehat{\pi}}

\newcommand{\sot}{\textsc{sot}}

\begin{document}
%%%%%%%%%%%%%%%%%%%%%%%%%%%%%%%%%%%%%%
\title[Semicrossed products]{Semicrossed products of operator algebras and their $\ca$-envelopes}

\author[E. Kakariadis]{Evgenios Kakariadis}
\address{Department of Mathematics\\University of Athens\\ 15784 Athens \\GREECE}
\email{mavro@math.uoa.gr}

\author[E.G. Katsoulis]{Elias~G.~Katsoulis}
\address{ Department of Mathematics\\University of Athens\\ 15784 Athens \\GREECE \vspace{-2ex}}
\address{\textit{Alternate address:} Department of Mathematics\\East Carolina University\\ Greenville, NC 27858\\USA}
\email{katsoulise@ecu.edu}

\begin{abstract}
Let $\A$ be a unital operator algebra and let $\alpha$ be an automorphism of $\A$ that extends to a $*$-automorphism
of its $\ca$-envelope $\cenv (\A)$. In this paper we introduce the isometric semicrossed product $\A \times_{\alpha}^{\is} \bbZ^+ $ and we show that $\cenv(\A \times_{\alpha}^{\is} \bbZ^+ ) \simeq \cenv ( \A ) \times_{\alpha} \bbZ$. In contrast, the $\ca$-envelope of the familiar contractive
semicrossed product $\A \times_{\alpha} \bbZ^+ $ may not equal $\cenv ( \A ) \times_{\alpha} \bbZ$.
Our main tool for calculating $\ca$-envelopes for
semicrossed products is the concept of a relative semicrossed product of an operator algebra,
which we explore in the more general context of injective endomorphisms.

As an application, we extend the main result of \cite{DavKatsdilation} to tensor algebras of $\ca$-correspondences. We show that if $\T_{\X}^{+}$ is the tensor algebra of a $\ca$-correspondence $(\X, \fA)$ and $\alpha$ a completely isometric
automorphism of $\T_{\X}^{+}$ that fixes the diagonal elementwise, then the contractive semicrossed product satisfies $ \cenv(\T_{\X}^{+} \times_{\alpha} \bbZ^+  )\simeq \O_{\X} \times_{\alpha} \bbZ$, where $\O_{\X}$ denotes the Cuntz-Pimsner algebra of $(\X, \fA)$.
\end{abstract}

\thanks{Second author was partially supported by a grant from ECU}

\date{}
\maketitle
%%%%%%%%%%%%%%%%%%%%%%%%%%%%%%%%%%%%%%%%%%%

%%%%%%%%%%%%%%%%%%%%%%%%%%%%%%%%%%%%%%%%%%%

\section{Introduction and preliminaries}        \label{Intro}

In this paper, we offer three choices for defining the semicrossed product of an operator algebra $\A$ by a unital, completely contractive endomorphism $\alpha$ of $\A$ (Definitions \ref{def:sem} and \ref{rel}.) In all cases, the resulting algebras contain a completely
isometric copy of $\A$ and a "universal" operator that implements the covariance relations.
In the case where $\A$ is a $\ca$-algebra and $\alpha$ preserves adjoints, all three choices
 produce the same operator algebra, Peters' semicrossed product of a $\ca$-algebra \cite{Pet} by an endomorphism. (Semicrossed products of $\ca$-algebras have been under investigation by various authors \cite{AlaP, Arv, ArvJ, DeAPet,DKconj, DKsurvey, HadH, MM, Pet2}, starting with the work of Arveson \cite{Arv} in the late sixties.) In the general
(non-selfadjoint) case however, the semicrossed products we introduce here may lead to non-isomorphic operator algebras.
The main objective of this paper
is to clarify the relation between
the three semicrossed products and calculate their $\ca$-envelope, whenever possible.

The present paper is a continuation of the recent work of Davidson and the second named author in \cite{DavKatsdilation}. In the language of the present paper, the main objective of \cite{DavKatsdilation} was to show that in the special case where $\A$ is Popescu's non-commutative disc algebra \cite{Pop3} and $\alpha$ a completely isometric automorphism of $\A$, all three semicrossed products coincide. One of the main results of this paper, Theorem \ref{main}, shows that two of these semicrossed products, the isometric and the relative one, coincide for \textit{any} operator algebra $\A$ and \textit{any} completely isometric automorphism $\alpha$ of $\A$. To prove this, we had to abandon the rather intricate but ad-hoc arguments of the second half of \cite{DavKatsdilation} and instead adopt an abstract approach. Theorem \ref{main} focuses now any further research on semicrossed products to the study of the other two, the isometric and the contractive semicrossed product. For these two, there are examples to show that they do not coincide in general (Remark \ref{counter}). Nevertheless, with Theorem \ref{main} in hand, we show that they do coincide in the case of a tensor algebra of a $\ca$-correspondence and and a completely isometric isomorphism of the algebra that fixes its diagonal (Corollary \ref{main}). This not only generalizes the main result of \cite{DavKatsdilation} to a broader context but also paves the way for additional results of this kind to come in the future.

The various semicrossed products we define in this paper
are actually closed images of the following universal semicrossed product under concrete representations.

\begin{defn}   \label{def:sem}
Let $\alpha$ be a unital, completely contractive endomorphism of an operator algebra $\A$.
A contractive (isometric) covariant representation $(\pi, K)$ of $(\A,\alpha)$
is a completely contractive representation
$\pi$ of $\A$ on a Hilbert space $\H$ and a contraction (resp. isometry) $K \in \B(\H)$ so that
\[ \pi(A) K = K \pi(\alpha(A)) \qforal A \in \A .\]
The contractive (resp. isometric) semicrossed product $\A \times_{\alpha} \bbZ^+$ (resp. $\A \times_{\alpha}^{\is} \bbZ^+$) for the system $(\A, \alpha)$
is the
universal operator algebra generated by a copy of $\A$ and a contraction (resp. isometry) $\fV$
so that $A\fV= \fV \alpha(A)$, for all $A \in \A$.
\end{defn}

The contractive
semicrossed product has, by definition, a rich representation theory which unfortunately makes it very intractable.
This was first observed in \cite{DavKatsdilation} based on the famous example of Varopoulos \cite{Var} regarding
three commuting contractions that do not satisfy the usual von Neumann inequality. Nevertheless, there are significant cases where the
contractive semicrossed product has been completely identified. These include the case where $\A$ is a $\ca$-algebra \cite{Kakar, Pet} and the case where
$\A$ is the non-commutative disc algebra $\fA_n$ and $\alpha$ is an isometric automorphism \cite{DavKatsdilation}.

The isometric semicrossed product is the (closed) image of $\A \times_{\alpha} \bbZ^+$ under the representation which restricts to the entries
where the contractions $K$ are actually isometries. We believe that this is
a more tractable object and as we shall see, in the case where $\alpha$ is a completely isometric automorphism, i.e., it extends to an automorphism
of the $\ca$-envelope $\cenv (\A)$  of $\A$, we can identify the $\ca$-envelope of $\A \times_{\alpha} \bbZ^+$ as the crossed product
$\ca$-algebra $\cenv ( \A ) \times_{\alpha} \bbZ$. The main tool for establishing this result is the concept of a relative semicrossed
product.

Recall that a $\ca$ algebra $\C$ is said to be a \textit{$\ca$-cover} for a subalgebra $\A \subseteq \C$
provided that $\A$ generates $\C$ as a $\ca$-algebra, i.e.,
 $\C = C^*(\A)$. If $\C$ is a $\ca$-cover for $\A$,
 then $\J_{\A}$ will denote the \v{S}ilov ideal of $\A$ in $\C$. Therefore,
 $\cenv (\A)=\C/\J_{\A}$ and the restriction of the natural projection
 $q:\C\rightarrow \C/\J_{\A}$ on $\A$ is a completely isometric representation of $\A$.
 (Any ideal $\J \subseteq \C$, with the property that the restriction of
 the natural projection $\C\rightarrow \C/\J$ on $\A$ is a complete isometry, is called a \textit{boundary ideal}
 and $\J_{\A}$ is the largest such ideal.)

\begin{defn}   \label{rel}
Let $\A$ be an operator algebra, $\C$ a $\ca$-cover of $\A$ and let $\alpha$ be an $*$-endomorphism of $\C$ that leaves $\A$ invariant. The
subalgebra of Peters' semicrossed product $\C \times_{\alpha} \bbZ^+$, which is generated by $\A \subseteq \C \subseteq  \C \times_{\alpha} \bbZ^+$ and
the universal isometry $\fV \in \C \times_{\alpha} \bbZ^+$, is denoted by $\A \times_{\C , \alpha} \bbZ^+$ and is said to be a relative semicrossed product for the system $(\A, \alpha)$.
\end{defn}

Therefore, the relative semicrossed product $\A \times_{\C \, , \alpha} \bbZ^+$ comes from the representation of $\A \times_{\alpha} \bbZ^+$ that restricts
to the entries where $\pi$ and $\alpha$ are $*$-extendable to $\C$ and the contraction $K$ satisfies the covariance relation with these extensions.
It seems plausible that non-isomorphic $\ca$-covers for $\A$ and varying extensions for the endomorphism $\alpha$ could produce non-isomorphic
relative semicrossed products. It turns out that under a reasonable technical requirement, i.e., invariance of the Shilov ideal, all such relative semicrossed products are
completely isometrically isomorphic (Proposition \ref{cis}). In particular, this requirement is satisfied when $\alpha$ is a completely isometric automorphism of $\A$; in that case all relative semicrossed products for $(\A, \alpha)$ are completely isometrically isomorphic to each other.

\section{The relative semicrossed product and its $\ca$-envelope}

We begin this section with some preliminary results. The first one is a standard result that shows how to lift an injective $*$-endomorphism of a $\ca$-algebra to an automorphism of a possibly larger $\ca$-algebra.

\begin{prop}\label{lift_end}
  If $\alpha$ is an injective endomorphism of a C*-algebra $\fA$,
  then there is a unique triple $(\fB,\beta,j)$ (up to isomorphism)
  where $\fB$ is a C*-algebra, $\beta$ is an automorphism of $\fB$
  and $j$ is a $*$-monomorphism of $\fA$ into $\fB$ such that
  $\beta j = j \alpha$ and
  $\fB = \ol{ \bigcup_{k\ge0} \beta^{-k} j(\fA )}$.

  To paraphrase, there is a unique minimal C*-algebra $\fB$
  containing $\fA$ with an automorphism $\beta$ satisfying
  $\beta|_\fA = \alpha$.
  \end{prop}

  \begin{proof}
   Consider the inductive limit $\fB$ of the system
  \[ \begin{CD}
 \fA_1 @> \alpha_1 >> \fA_2 @> \alpha_2 >> \fA_3 @> \alpha_3 >> \fA_4
 @> \alpha_4 >> \cdots ,
\end{CD}\]
where $\fA_i=\fA$ and $\alpha_i = \alpha$, for all $i \in \bbN$.
Let $j_i$ be the associated $*$-monomorphism from $\fA = \fA_i$ to $\fB$. This map is defined
as $j_i(A)=
(0,0,\dots , 0, A, \alpha (A), \alpha^{2} (A), \dots)$, with the
understanding that the infinite tuple in the definition
signifies the appropriate equivalence class. Define $j=j_1$.

The system
\[
\begin{CD}
 \fA_1 @> \alpha_1 >> \fA_2 @> \alpha_2 >> \fA_3 @> \alpha_3 >> \cdots \\
 @V\alpha VV            @V\alpha VV              @V\alpha VV                   \\
 \fA_1 @> \alpha_1 >> \fA_2 @> \alpha_2 >> \fA_3 @> \alpha_3 >> \cdots
\end{CD}
\]
gives rise to an $*$-automorphism $\beta$ of $\fB$ defined as
\[
\beta (A_1 , A_2 , A_3 , \dots ) =  ( \alpha (A_1) , \alpha (A_2) , \alpha (A_3) , \dots ), \quad
A_i \in \fA_i  , i \in \bbN.
\]
Clearly, $\beta j = j \alpha$. The inverse
of $\beta$ on $\bigcup_{k\ge1} j_k (\fA_k)$ satisfies
\[
\beta^{-1} (A_1 , A_2 , A_3 , \dots ) =  (0, \alpha (A_1) , \alpha (A_2) , \alpha (A_3) , \dots ), \quad
A_i \in \fA_i  , i \in \bbN.
\]
and so, if $A \in \fA$, then
\[
\beta^{-k} (A , 0 , 0 , \dots ) =  (0, 0, \dots , 0, A,  \alpha (A) , \alpha^{2} (A) , \dots).
\]
Therefore, $\bigcup_{k\ge0} \beta^{-k} j(\fA)$ is dense in $\fB$.
\end{proof}

We now fix some notation and use the previous result to construct a useful embedding of
$\fA \times_{\alpha} \bbZ^+$.

Let $\A$ be an operator algebra and let $\alpha$ be a completely contractive endomorphism of $\A$.
If $\pi$ is a completely contractive
representation of $\A$ on a Hilbert space $\H$, we define
 \[
 \widetilde{\pi}:\A \longrightarrow \B(\H\otimes\ell^2(\bbZ^+))
 \]
  so that
  \begin{equation} \label{pitilde}
  \widetilde{\pi}(A)\equiv (\pi(A), \pi(\ga(A)), \pi(\ga^2(A)), \dots),\quad A \in \A.
  \end{equation}
  Let $V_{\H} \equiv I\otimes V$, where $V$ denotes the unilateral
  shift on $\ell^2(\bbZ^+)$. The pair $( \widetilde{\pi}, V_{\H})$
  forms a contractive covariant representation of $(\A, \alpha)$ and the associated representation of
  $\A \times_{\alpha} \bbZ^+$ is denoted as $\widetilde{\pi} \times V_{\H}$.
 If $\ga$ happens to be a completely isometric automorphism of $\A$, we also
  have the representation $\widehat{\pi}:\A\rightarrow \B(H\otimes
  \ell^2(\bbZ))$, such that
  $$\widehat{\pi} (A)\equiv(\dots, \pi(\ga^{-1}(A)), \pi(A), \pi(\ga(A)), \pi(\ga^2(A)), \dots),\quad A \in \A,$$
   the unitary $U_{\H}=I\otimes U$, where $U$ is the bilateral
  shift on $\ell^2(\bbZ)$ and the associated representation $\widehat{\pi} \times U_{\H}$
  of $\fA \times_{\alpha} \bbZ^+$.

\begin{prop} \label{embed}
 Let $\alpha$ is an injective endomorphism of a C*-algebra $\fA$
 and let $(\fB,\beta,j)$ be the triple of Proposition \ref{lift_end}.
 Then $\fA \times_{\alpha} \bbZ^+$ embeds completely isometrically
 in $\fB \times_{\beta} \bbZ$. Furthermore, $\fB \times_{\beta} \bbZ$ becomes
 a $\ca$-cover for $\fA \times_{\alpha} \bbZ^+$.
 \end{prop}

 \begin{proof}
 Let $\pi$ be a faithful representation of $\fA$ on a Hilbert space $\H$.
 Since every representation of $\fA$ is a direct sum of cyclic representations, the GNS construction implies
 that there exists a representation $\pi_{\beta}$ of $\fB$ on a Hilbert space $\H_{\beta} \supseteq \H$ so that
 $\H $ is reducing for $\pi_{\beta}(j(\fA))$ and $\pi_{\beta}(j(A))\mid_{ \H}= \pi(A)$ for all $A \in \fA$.

 By gauge invariance, $\widetilde{\pi} \times V_{\H}$ is a completely isometric representation for $\fA \times_{\alpha} \bbZ^+$;
 therefore the same is true for the representation $\widehat{\pi_{\beta}}\, j\times U_{\H_{\beta}}$.
 Now notice that the representation $\widehat{\pi_{\beta}}$ is faithful on
  $\bigcup_{k\ge0} \beta^{-k} j(\fA)$ and so,  by inductivity, on all of $\fB$. By gauge invariance, the representation
  $\widehat{\pi_{\beta}}\times U_{\H_{\beta}}$ is also faithful on $\fB \times_{\beta} \bbZ$. The
  proposition now follows by comparing the ranges of $\widehat{\pi_{\beta}} \, j\times U_{\H_{\beta}}$ and
  $\widehat{\pi_{\beta}}\times U_{\H_{\beta}}$.
   \end{proof}

   As we shall see in Theorem \ref{cstar},
   $\fB \times_{\beta} \bbZ$ is actually the $\ca$-envelope of $\fA \times_{\alpha} \bbZ^+$.

Let $\A$ be an operator algebra and let $\C$ be a $\ca$-cover of $\A$.
Let $\alpha$ be a $*$-endomorphism of $\C$ that leaves invariant both
$\A$ and $\J_{\A}$ and let $\dot{\ga}:\C/\J_{\A}  \rightarrow \C/\J_{\A}$ be defined
as $\dot{\ga}(X+\J_{\A})=\ga(X)+\J_{\A}$, $X \in \C$. In this context, there are two relative semicrossed products
to be considered, $\A \times_{\C , \alpha} \bbZ^+$ and $\A/\J_{\A} \times_{\C /\J_{\A}, \dot{\alpha}} \bbZ^+$.
The following proposition, which clarifies the relation between these two semicrossed product,
is an application of two significant recent results in
the theory of maximal dilations for completely contractive maps.
First, Dritschel and McCullough \cite{DritsM} have recently proven that every
completely contractive map $\phi : \A \rightarrow B(\H)$ admits a maximal dilation
$(\Phi , \K)$, i.e., a dilation $\Phi : \A \rightarrow B(\K)$ so that any further dilation of $\Phi$ has $\Phi$ as a direct summand.
Furthermore, Muhly and Solel \cite{MS} have shown
that any such maximal dilation $\Phi$ extends (uniquely) to a $*$-representation of any $\ca$-cover of $\A$.

\begin{prop}\label{cis}
Let $\A$ be an operator algebra, $\C$ be a $\ca$-cover of $\A$
and let $\alpha$ be a $*$-endomorphism of $\C$ that leaves invariant both
$\A$ and $\J_{\A}$. Then the relative semicrossed products $\A \times_{\C , \alpha} \bbZ^+$
and  $\A/\J_{\A} \times_{\C /\J_{\A}, \dot{\alpha}} \bbZ^+$ are completely
isometrically isomorphic.
\end{prop}

\begin{proof}
 Let $F=\sum_{n=0}^k \, \fV^n A_n \in \A \times_{\C , \alpha} \bbZ^+$
and $F'=\sum_{n=0}^k \, \fV^{n} (A_n + \J_{\A}) \in \A/\J_{\A} \times_{\C /\J_{\A}, \dot{\alpha}} \bbZ^+$.
We have to show that the homomorphism $F\mapsto  F'$ is a
completely isometric map.

Let $\pi$ be a faithful representation of $\C$ on a Hilbert space $\H$ and let $(\widetilde{\pi}, V_{\H})$ be as in the beginning of the section (see (\ref{pitilde})).
Consider the completely isometric map
\[
\phi:\A/\J_{\A}  \longrightarrow \B(\H): A + \J_{\A} \longmapsto \pi(A), \quad A\in \A.
\]
According to our earlier discussion, there is a
maximal dilation $(\Phi, \K)$ of $\phi$ which extends uniquely to a
representation of $\C/\J_{\A} $ such that
\[
P_{\H}\Phi(A+\J_{\A} )|_{\H}=\phi(A+\J_{\A} )=\pi(A),
\]
for all $A \in A$. Since $P_{\H\otimes\ell^2(\bbZ_+)}=
P_{\H}\otimes I$, we have that
\[
P_{\H \otimes\ell^2(\bbZ_+)}
\widetilde{\Phi}(A+\J_{\A} ))|_{\H\otimes\ell^2(\bbZ_+)}=\widetilde{\pi}(A+\J_{\A} ),
\]
for all $A \in A$. Also, $V_{\K}|_{\H\otimes\ell^2(\bbZ_+)}=V_{\H}$ and so
\begin{align*}
\|F\| &=\|\sum_{n=0}^k \, V^{n}_{\H} \widetilde{\pi}(A_n) \|   \\
&=\|P_{\H\otimes\ell^2(\bbZ_+)}\left(\sum_{n=0}^k \,
V^{n}_{\K} \widetilde{\Phi}(A_n + \J_{\A}) \right)|_{\H\otimes\ell^2(\bbZ_+)}\|\\
& \leq \| \sum_{n=0}^k \, V^{n}_{\K} \widetilde{\Phi}(A_n + \J_{\A}) \| \leq \|F'\|.
\end{align*}
The same is also true for all the matrix norms and so the map $F'\mapsto  F$ is well defined and completely contractive.
By reversing the roles of $\A$ and $\A/\J(\A)$ in
the previous arguments, we can also prove that $F \mapsto  F'$ is completely contractive,
and the conclusion follows.
\end{proof}

Now we wish to identify the $C^*$-envelope of $\A \times_{\C , \alpha} \bbZ^+$. From
the previous result we know that it coincides with the
$C^*$-envelope of $\A/\J_{\A} \times_{\C /\J_{\A}, \dot{\alpha}} \bbZ^+$. In the following
we consider the case where $\ga:\C\rightarrow \C$ is injective.
This is easily seen to imply that $\dot{\ga}:\C/\J_{\A}\rightarrow \C/\J_{\A}$ is an injective
$*$-homomorphism. Indeed,

\begin{lem}
Let $\A$ be an operator algebra, $\C$ be a $\ca$-cover of $\A$
and let $\alpha$ be an injective $*$-endomorphism of $\C$ that leaves invariant both
$\A$ and $\J_{\A}$. Then  $\dot{\ga}:\C/\J_{\A}\rightarrow \C/\J_{\A}$ is an injective
$*$-homomorphism.
\end{lem}

\begin{proof}
In that case, $\ga$ is a
completely isometric map. Therefore,
\[
\nor{q(A)+\ker\dot{\ga}}=\nor{\dot{\ga}(A+\J_{\A})} =\nor{\ga(A)+ \J_{\A}}=
\nor{\ga(A)}=\nor{A},
\]
since $\J_{\A}$ is a boundary ideal and $\ga(\A)\subseteq\A$. The same
argument holds for all the matrix norms. Thus $\ker\dot{\ga}$ is a
boundary ideal of $\C/\J_{\A}$. However, $\C/\J_{\A}$ is the $\ca$-envelope
of $\A$ and so it contains no non-trivial boundary ideals for $\A$.
Thus $\ker\dot{\ga}=(0)$.
\end{proof}

The following is the main technical result of the section.

\begin{thm}\label{cstar}
Let $\A$ be an operator algebra, $\C$ be a $\ca$-cover of $\A$  and let $\J_{\A}$ be the
\v{S}ilov ideal of $\A$ in $\C$. Let $\alpha$ be an injective $*$-endomorphism of $\C$ that leaves invariant both
$\A$ and $\J_{\A}$. Then
\[
\cenv(\A \times_{\C , \alpha} \bbZ^+  )\simeq \fB \times_{\beta} \bbZ\, ,
\]
where $(\fB, \beta, j)$ is the unique triple of Proposition \ref{lift_end} associated with the injective
$*$-endomorphism $\dot{\alpha}$ of $\C/\J_{\A}$.
\end{thm}

\begin{proof} Proposition  \ref{cis} shows that it suffices to identify
the $C^*$-envelope of $\A/\J_{\A} \times_{\C /\J_{\A}, \dot{\alpha}} \bbZ^+$.

If $(\fB, \beta, j)$ is the unique triple of Proposition \ref{lift_end} associated with the injective
$*$-endomorphism $\dot{\alpha}$ of $\C/\J_{\A}$ then Proposition \ref{embed} shows that
$\C/\J_{\A} \times _{\dot{\alpha}} \bbZ^+$, and therefore $\A/\J_{\A} \times_{\C /\J_{\A}, \dot{\alpha}} \bbZ^+$,
embeds completely isometrically in  $\fB \times_{\beta} \bbZ$. Moreover,
$\fB \times_{\beta} \bbZ$ is a $\ca$-cover for $\A/\J_{\A} \times_{\C /\J_{\A}, \dot{\alpha}} \bbZ^+$.
Let $\J$ be the \v{S}ilov ideal of $\A/\J_{\A} \times_{\C /\J_{\A}, \dot{\alpha}} \bbZ^+$ in
$\fB \times_{\beta} \bbZ$. We are to show that $\J=\{0\}$.

Assume to the contrary that $\J \neq \{ 0\}$. Since $\J$ is invariant by automorphisms of the $\ca$-cover, it remains
invariant by the natural gauge action on $\fB \times_{\beta} \bbZ$. Therefore it has non-trivial intersection
with the fixed point algebra of the natural gauge action, i.e., $\J \cap \fB \neq \{ 0\}$. However
\[
\fB = \ol{ \bigcup_{k\ge0} \beta^{-k} j(\C/\J_{\A} )}
\]
and therefore by inductivity there exists $k \in \bbN$ so that
\[
 \J \cap \beta^{-k} j(\C/\J_{\A} ) \neq \{0\}.
 \]
However, $\beta$ acts by conjugating with a unitary in $\fB \times_{\beta} \bbZ$. Since $\J$ is an ideal of
$\fB \times_{\beta} \bbZ$, the above implies that
\[
 \J \cap  j(\C/\J_{\A} ) \neq \{0\}.
 \]
 But then $ j^{-1}\left( \J \cap  j(\C/\J_{\A} )\right)$ is a non-zero boundary ideal for $\A$ in
 $\C/\J_{\A}$, a contradiction.
 \end{proof}

 In \cite{DavKatsdilation} Davidson and the second named author proved that
 \[
 \cenv(\A_n \times_{\alpha} \bbZ^+) =\O_n \times_{\alpha} \bbZ,
 \]
 where $\A_n$ is Popescu's non-commutative disc algebra \cite{Pop3}, $\alpha$ a is a (completely) isometric automorphism of $\A_n$ and
 $\O_n$ denotes the Cuntz algebra generated by $n$ isometries. A dilation result in the first half of \cite{DavKatsdilation}
 reduces the problem of calculating the $\ca$-envelope of $\A_n \times_{\alpha} \bbZ^+$ to essentially verifying
 that $\cenv(\A_n \times_{\alpha}^{\is} \bbZ^+) =\O_n \times_{\alpha} \bbZ$.
 It takes the second half of \cite{DavKatsdilation} and intricate use of the representation theory for $\O_n$ to verify that claim.
 The next result establishes
  the same claim for \textit{arbitrary} operator algebras using only abstract arguments.

 \begin{thm}  \label{main}
Let $\A$ be an operator algebra and $\alpha$ be an automorphism of $\A$ that extends to a $*$-automorphism of
$\cenv(\A)$. Then, any relative semicrossed product for $(\A, \alpha)$ is completely isometrically isomorphic to $\A \times_{\alpha}^{\is} \bbZ^+ $. Hence,
\[
\cenv(\A \times_{\alpha}^{\is} \bbZ^+  )\simeq \cenv (\A) \times_{\alpha} \bbZ.
\]
\end{thm}

\begin{proof}

In light of Theorem \ref{cstar}, it suffices to show that $\A \times_{\alpha}^{\is} \bbZ^+  $ dilates to a relative
semicrossed product. This is done as follows.

Let $\fV \in \A \times_{\alpha}^{\is} \bbZ^+ $ be the universal isometry acting on a Hilbert space $\fH$
and let $\H$ be the direct limit Hilbert space of the inductive system
 \[ \begin{CD}
 \fH @> \fV >> \fH @> \fV >> \fH @> \fV >> \cdots .
\end{CD}\]
For each $A \in \A$,
the commutative diagram
\[
\begin{CD}
 \fH @> \fV>> \fH @> \fV >> \fH @> \fV >> \cdots \\
 @VA VV            @V\alpha^{-1}(A) VV              @V\alpha^{-2}(A) VV                   \\
 \fH @> \fV >> \fH @> \fV >> \fH @> \fV >> \cdots
\end{CD}
\]
defines an operator $\pi(A) \in \B(\H)$. It is easily seen that $\pi$ defines a completely isometric representation of $\A$
on $\H$.
Consider now the unitary $U \in \B(\H)$ defined as
\[
U(h_1 , h_2 , h_3 , \dots ) =  ( h_2 , h_3 , \dots ), \quad
h_i \in \fH  , i \in \bbN.
\]
and notice that $\pi(\alpha(A))= U^* \pi(A)U$, $A \in \A$. Therefore, the conjugation by $U$
defines a $*$-automorphism of $\C \equiv \ca(\pi(\A))$, which extends $\alpha$ and is denoted by the same symbol as well. Therefore,
 \[
 \A \times_{\alpha}^{\is} \bbZ^+ \simeq \A \times_{\C, \alpha} \bbZ^+
 \]
 and the conclusion follows from Theorem \ref{cstar}.
\end{proof}

\begin{rem} In light of Theorem \ref{main}, we wonder whether one can compute the $\ca$-envelope of $\A \times_{\alpha}^{\is} \bbZ^+$ in the case where $\alpha$ is an endomorphism of $\A$ that extends to an injective $*$-endomorphism of $\cenv(\A)$. To do this, one will have to prove an analogue of Theorem \ref{cstar} in the case where $\alpha$ may not preserve the \v{S}ilov ideal $\J_{\A}$ of $\A$ in $\C$.
\end{rem}

\begin{rem} \label{counter}
An observation from \cite{DavKatsdilation} shows that Theorem \ref{main} fails for the contractive
semicrossed product, thus showing that the isometric and the contractive semicrossed product are not completely isometrically isomorphic in general.

 Indeed, the bidisk
algebra $A(\bbD^2)$ sits inside $\rC(\bbT^2)$, which is its C*-envelope
by Ando's theorem.  Consider the identity automorphism $\id$.
Ando's theorem also shows that the completely contractive representations
of $A(\bbD^2)$ are determined by an arbitrary pair $T_1,T_2$ of
commuting contractions.  A covariant representation of $(A(\bbD^2),\id)$
is given by such a pair and a third contraction $T_3$ which commutes
with $T_1$ and $T_2$.  If it were true that the C*-envelope of this system
was $\rC(\bbT^2) \times_\id \bbZ \simeq \rC(\bbT^3)$, then it would be true
that every commuting triple of contractions satisfies the usual von Neumann inequality.
This has been disproved by Varopoulos \cite{Var}.
\end{rem}

\section{An application to tensor algebras}

In spite of Remark \ref{counter}, there are special cases where the contractive and isometric semicrossed products coincide. The purpose of this section is to verify this in the case where $\A$ is the tensor algebra of a $\ca$-correspondence and $\alpha$ a completely isometric isomorphism of $\A$ that fixes its diagonal elementwise (Corollary \ref{main}).

The tensor algebras for $\ca$-correspondences were introduced by Muhly and Solel in \cite{MS2}. This is a broad class of non-selfadjoint operator algebras which includes as special cases Peters' semicrossed products \cite{Pet2}, Popescu's non-commutative disc algebras \cite{Pop3}, the tensor algebras of graphs (introduced in \cite{MS2} and further studied in \cite{KaK, KrP}) and the tensor algebras for multivariable dynamics \cite{DavKatsMem}, to mention but a few.

Let $\fA$ be a $\ca$-algebra and $\X$ be a (right) Hilbert $\fA$-module, whose inner product is denoted as
$\langle \, . \mid . \, \rangle$. Let $\L (\X)$ be the adjointable operators
on $\X$ and let $\K(\X)$ be the norm closed subalgebra of $\L(\X )$ generated by the
operators $\theta_{\xi , \eta}$, $\xi , \eta \in \X$, where $\theta_{\xi , \eta}(\zeta)=
\xi \langle \eta | \zeta \rangle$, $\zeta \in \X$.

A Hilbert $\fA$-module $\X$ is said to be a \textit{$\ca$-correspondence} over $\fA$ provided that there exists
a $*$-homomorphism $\phi_{\X} : \fA \rightarrow \L(\X)$. We refer to $\phi_{\X}$ as the left action of $\fA$ on
$\X$.
%A $\ca$-correspondence $\X$ over $\fA$ is said to be \textit{faithful}
%if and only if the map $\phi_{\X}$ is faithful. A $\ca$-correspondence $\X$ over %$\fA$ is called \textit{strict} iff
%$\overline{[\phi_{\X}(\fA)\X ]} \subseteq \X$ is complemented, as a submodule
%of the Hilbert $\fA$-module $\X$. In particular, if $\overline{[ \phi_{\X}(\fA)\X %]} = \X$, i.e., the map
%$\phi_{\X}$ is non-degenerate, then $\X$ is said to be \textit{essential}.
From
a given $\ca$-correspondence $\X$ over $\fA$, one can form
new $\ca$-correspondences over $\fA$, such as the $n$-\textit{fold ampliation} or {direct sum}
$\X^{(n)}$ (\cite[page 5]{L})
and the $n$-\textit{fold interior tensor product} $\X^{\otimes n} \equiv
\X \otimes_{\phi_{\X}}  \X \otimes_{\phi_{\X}} \dots  \otimes_{\phi_{\X}}\X$ (\cite[page 39]{L}, $n \in \bbN$,
($\X^{\otimes 0}\equiv \fA$).
These operation are defined within the category of
$\ca$-correspondences over $\fA$. (See \cite{L} for more details.)

A \textit{representation} $(\pi , t)$ of a $\ca$-correspondence $\X$ over $\fA$ on a
$\ca$-algebra $\B$ consists of
a $*$-homomorphism $\pi : \fA \rightarrow \B$
and a linear map $t : \X \rightarrow \B$ so that
\begin{itemize}
\item[(i)] $t(\xi )^* t(\eta) = \pi(\langle \xi |\, \eta \rangle)$, for $\xi , \eta \in \X$,
\item[(ii)] $\pi (A) t(\xi)= t(\phi_{\X}(A) \xi)$, for $A \in \fA$, $\xi \in \X$.
\end{itemize}
For a representation $(\pi , t)$ of  a $\ca$-correspondence $\X$ there exists a $*$-homomorphism
$\psi_t : \K(\X ) \rightarrow \B$ so that $\psi_t (\theta_{\xi , \eta}) = t(\xi)t(\eta)^*$,
for $\xi , \eta \in \X$. Following Katsura \cite{Ka}, we say that the representation $(\pi , t)$ is
\textit{covariant} iff $\psi_t ( \phi_{\X} (A)) = \pi(A)$, for all $A \in \J_{\X}$, where
\[
\J_{\X} \equiv \phi_{\X}^{-1}(\K(\X)) \cap (\ker \phi_{\X})^{\perp} .
\]
If $(\pi , t)$ is a representation of $\X$ then the $\ca$-algebra (resp. norm closed algebra) generated by the images of
$\pi$ and $t$ is denoted as $\ca(\pi , t)$ (resp. $\Alg((\pi , t)$). There is a
universal representation $(\overline{\pi}_{\fA}, \overline{t}_{\X})$
for $\X$ and the
$\ca$-algebra $\ca (\overline{\pi}_{\fA}, \overline{t}_{\X})$ is the Toeplitz-Cuntz-Pimsner
algebra $\T_{\X}$. Similarly, the
Cuntz-Pimsner algebra $\O_{\X}$ is the $\ca$-algebra generated by the image of the universal
covariant representation $(\pi_{\fA}, t_{\X})$ for $\X$.

A concrete presentation of both $\T_{\X}$ and $\O_{\X}$ can be given in terms of the generalized
Fock space $\F_{\X}$ which we now describe. The \textit{Fock space} $\F_{\X}$ over the
correspondence $\X$ is defined to be the direct sum of the $\X^{\otimes n}$ with the structure of a direct sum
of $\ca$-correspondences over $\fA$,
\[
\F_{\X}= \fA \oplus \X \oplus \X^{\otimes 2} \oplus \dots .
\]
Given $\xi \in \X$, the (left)
creation operator $t_{\infty}(\xi) \in \L(\F_{\X})$ is defined by the formula
\[
t_{\infty}(\xi)(A , \zeta_{1}, \zeta_{2}, \dots ) = (0, \xi A, \xi \otimes \zeta_1,
\xi \otimes \zeta_2, \dots),
\]
where $\zeta_n \in \X^{\otimes n}$, $n \in \bbN$, and $A \in \fA$. Also, for $A \in \fA$, we define
$\pi_{\infty}(A)\in \L(\F_{\X})$ to be the diagonal operator with $\phi_{\X}(A)\otimes id_{n-1}$
at its $\X^{\otimes n}$-th entry. It is easy to verify that $( \pi_{\infty}, t_{\infty})$ is a representation of
$\X$ which is called the \textit{Fock representation} of $\X$. Fowler and Raeburn \cite{FR} (resp. Katsura
\cite{Ka}) have shown that
the $\ca$-algebra $\ca ( \pi_{\infty}, t_{\infty})$
(resp $\ca ( \pi_{\infty}, t_{\infty})/ \K(\F_{\X\J_{\X}})$)
is isomorphic to $\T_{\X}$ (resp. $\O_{\X}$).

\begin{defn}
The \textit{tensor algebra} of a $\ca$-correspondence $\X$ over $\fA$ is the norm-closed algebra
$\alg(\overline{\pi}_{\fA}, \overline{t}_{\X})$ and is denoted as $\T_{\X}^{+}$.
\end{defn}

According to \cite{FR, Ka}, the algebras $\T_{\X}^{+} \equiv \alg(\overline{\pi}_{\fA}, \overline{t}_{\X})
$ and $ \alg( \pi_{\infty}, t_{\infty} )$ are completely isometrically isomorphic and we will
therefore identify them.

In order to prove the main results of this section, we follow the strategy of the first half of \cite{DavKatsdilation}. However, the generality in which we are working with, presents new difficulties and requires innovation. One such innovation is the following.

\begin{lem} \label{Solel}
Let $(\X, \fA)$ be a $\ca$-correspondence, let $\alpha$ be a completely isometric
automorphism of the associated tensor algebra $\T_{\X}^{+}$ and assume that
$\alpha(A)=A$, for all $A \in \fA$. Let $\pi: \T_{\X}^{+} \rightarrow B(\H)$
be a completely contractive representation of $\T_{\X}^{+}$ and let $X \in B(\H)$ be a contraction satisfying,
\[
\pi(L)X = X \pi(\alpha(L)), \mbox{ for all } L \in \T_{\X}^{+}.
\]
Then there exist isometric co-extensions $\pi'$ and $(\pi\circ \alpha)'$, of
$\pi$ and  $\pi\circ \alpha$ respectively, and an isometric co-extension $X'$ of $X$, all acting on some Hilbert space $\H'$ and satisfying
\begin{equation} \label{eq;Solel1}
\pi'(L)X'=X'(\pi\circ \alpha)'(L), \mbox{ for all } L \in \T_{\X}^{+},
\end{equation}
and
\begin{equation} \label{eq;Solel2}
\pi'(A)=(\pi \circ \alpha)'(A), \mbox{ for all } A \in \fA.
\end{equation}
\end{lem}

\begin{proof} First we construct isometric co-extensions $\hpi_1$ and $\hpi_2$, of $\pi$ and $\pi \circ \alpha$ respectively, and an isometric co-extension $\widehat{X}$ of $X$, with the property that
\begin{equation} \label{agree}
\hpi_1(A)=\hpi_2(A)
\end{equation}
and
\begin{equation} \label{commute}
\phantom{XXiXXX} \widehat{X} \hpi_i(A) = \hpi_i (A)\widehat{X}, \quad i=1,2,
\end{equation}
for all $A \in \fA$.

To do this, notice that $X$ commutes with $\pi(\fA)$. co-extend $X$ to its Schaeffer dilation
\[
 S_{X} \simeq
  \begin{bmatrix}
   K &  0 &  0 & 0 &  \dots \\
   D_{K}  &0 &  0  & 0 &  \dots \\
   0  &  I  & 0 &  0 & \dots \\
   0  &  0  & I & 0  &  \dots   \\
   \vdots & \vdots & \vdots & \vdots  & \ddots
   \end{bmatrix}
   \in
   B(\H^{(\infty)}),
\]
where $D_{K} = (I-K^{*}K)^{1/2}$. Let $\pi^{(\infty)}$ be the infinite ampliation of $\pi$ and notice that $S_X$ commutes with $\pi^{(\infty)}(\fA)$.
Subsequently, using \cite[Theorem 3.3]{MS2}, we obtain some isometric co-extension $\hpi$ of $\pi^{(\infty)}$, on some Hilbert space $\K=\H^{(\infty)} \oplus \M$, and let $\hpi_1 = \hpi$, $\hpi_2 = \hpi \circ \alpha$ and $\widehat{X} = S_{X} \oplus I_{\M}$. These $\hpi_1, \hpi_2$ and $\widehat{X}$ satisfy (\ref{agree}) and (\ref{commute}).

Since $\widehat X$ satisfies (\ref{commute}), the pairs $( \widehat t _i, \hpi_i | _{\fA})$, $i=1,2$, where,
\begin{align*}
\widehat t_1(\xi) &= \hpi_1 ( t_{\infty} (\xi) )\widehat X, \\
\widehat t_2(\xi) &= \widehat X \hpi_2 (t_{\infty} (\xi)), \quad \xi \in \X,
\end{align*}
define isometric representations of $(\X, \fA)$ and so there exist \break $*$-representations $\rho_i: \T_{\X} \rightarrow B(\K)$ which integrate $( \widehat t _i, \hpi_i | _{\fA})$, $i=1,2$. Since,
\begin{equation} \label{equal}
P_{\H}\rho_1(L)|_{\H}= P_{\H}\rho_2(L)|_{\H}, \foral L \in \T_{\X}^{+},
\end{equation}
the representations $\rho_i$ co-extend the same contractive representation of $\T_{\X}^{+}$ (appearing in (\ref{equal})). By the uniqueness of the minimal isometric co-extension \cite[Proposition 3.2]{MS2}, there exist projections $Q_i$ commuting with $\rho_i(\T_{\X}^{+})$, $i=1,2$, (hence commuting with $\hpi_i (\fA)$, $i=1,2$) and a unitary $W: Q_1(\K) \rightarrow Q_2(\K)$, so that
\begin{equation} \label{Wunitary}
W\rho_1(L)|_{Q_1(\K)}W^{*}= \rho_2(L)|_{Q_2(\K)}, \foral L \in \T_{\X}^{+}.
\end{equation}
Furthermore, for each $i=1,2$, $\H \subseteq Q_i(\K)$ and $W$ fixes $\H$, because both $\rho_1$ and $\rho_2$ co-extend the same completely contractive representation of $\T_{\X}^{+}$, which acts on $\H$.

For $i=1,2$, let
\[
\tpi_i(L) =\hpi_i(L)\oplus \left(\rho_1(L)|_{Q_1^{\perp}(\K)}\right)^{(\infty)} \oplus \left(\rho_2(L)|_{Q_2^{\perp}(\K)}\right)^{(\infty)}, \,\, L \in \T_{\X}^{+},
\]
and let
\[
\widetilde X = \widehat X \oplus I_{Q_1^{\perp}(\K)}^{(\infty)} \oplus I_{Q_2^{\perp}(\K)}^{(\infty)},
\]
all of them acting on
\[
\H' = \K\oplus Q_1^{\perp}(\K)^{(\infty)} \oplus Q_2^{\perp}(\K)^{(\infty)}.
\]
Because of (\ref{Wunitary}), there exists a unitary $U \in B(\H')$ which fixes $\H$, commutes with $\tpi_1(\fA)=\tpi_2(\fA)$ and satisfies
\[U\tpi_1(L)\widetilde XU^{*} = \widetilde X \tpi_2(L), \foral L \in \T_{\X}^{+}.
\]
Consider the isometric representations $(t_i , \tpi_i |_{\fA})$ of $(\X, \fA)$, where,
\begin{align*}
t_1(\xi) &= \tpi_1 ( t_{\infty} (\xi) )U, \mbox{ and,}\\
t_2(\xi) &=  \tpi_2 (t_{\infty} (\xi))U, \xi \in \X,
\end{align*}
and let $\pi_i'$, $i=1,2$, be the $*$-representations of $ \T_{\X}$ which integrate them.
Let $X' = U^* \widetilde X$ and notice that for any $L \in \T_{\X}^{+}$ we have
\[
\pi_1'(L)X'= \hpi_1(L)UU^* \widetilde X =\hpi_1(L)\widetilde X
\]
while
\[
X'\pi_2'(L) =U^*\widetilde X \hpi_2(L)U= \hpi_1(L)\widetilde X,
\]
and the conclusion follows.
\end{proof}

\begin{rem} \label{rem;Solel}
In the previous Lemma, one may take the isometric co-extension $X'$ to be the minimal isometric co-extension $X_m$ of $X$. In that case however, the co-extensions $\pi'$ and $(\pi\circ \alpha)'$ can only be considered completely contractive and not necessarily isometric.

Indeed, using Lemma \ref{Solel}, we obtain isometric co-extensions $\pi' , (\pi \circ \alpha)'$ and $ X'$ on some Hilbert space $\H'$ that satisfy (\ref{eq;Solel1}) and (\ref{eq;Solel2}).
Let $Q$ be the reducing projection for $X'$ so that $QX'|_{Q(\H')}\simeq X_m$. By
(\ref{eq;Solel2}), the projection $Q$ commutes with $\pi'(\fA)$ and $(\pi \circ \alpha)'(\fA)$ and so the completely contractive representations of $(\X, \fA)$, determined by the representations $\pi'$ and $(\pi \circ \alpha)'$ of $\fA$ and the mappings
\begin{align*}
\X \ni \xi &\longmapsto Q \pi'(t_{\infty}(\xi))|_{Q(\H')}\\
\X \ni \xi &\longmapsto Q (\pi\circ \alpha)'(t_{\infty}(\xi))|_{Q(\H')},
\end{align*}
can be integrated to the desired contractive representations of $ \T_{\X}^{+}$, satisfying the analogues of (\ref{eq;Solel1}) and (\ref{eq;Solel2}) with $X_m$ instead of $X'$.
\end{rem}

If we take $\alpha = \id$ in Lemma \ref{Solel}, then we obtain the commutant lifting Theorem of Muhly and Solel \cite{MS2}, without using the "one-step" extension in the proof. (In \cite[page 418]{MS2} the authors ask for such a proof.) Indeed

\begin{cor} \label{comlifting}
Let $(\X, \fA)$ be a $\ca$-correspondence, let $\pi: \T_{\X}^{+} \rightarrow B(\H)$
be a completely contractive representation of $\T_{\X}^{+}$ and let $X \in B(\H)$ be a contraction satisfying,
\[
\pi(L)X = X \pi(L), \mbox{ for all } L \in \T_{\X}^{+}.
\]
If $\pi_{m}$ is the minimal isometric co-extension of $\pi$, then there exists a contraction $X'$ co-extending $X$ and satisfying
\[
\pi_m(L)X'=X'\pi_m(L), \mbox{ for all } L \in \T_{\X}^{+}.
\]
\end{cor}

\begin{proof}
Use Lemma \ref{Solel} to obtain isometric co-extensions $\pi', X''$ on some Hilbert space $\H'$ that do the job. (Note however that $\pi'$ may be "larger" than the minimal isometric co-extension.) There exists now a reducing subspace $\K\subseteq \H'$ for $\pi'$ so that $\pi'|_{\K} \equiv \pi_m$. Letting $X'$ be the compression of $X''$ on $\K$, the conclusion follows.
\end{proof}

A familiar $2 \times 2$ matrix trick also establishes the intertwining form of  the commutant lifting theorem for minimal isometric co-extensions.

\begin{thm}
Let $(\X, \fA)$ be a $\ca$-correspondence, let $\alpha$ be a completely isometric
automorphism of the associated tensor algebra $\T_{\X}^{+}$ and assume that
$\alpha(A)=A$, for all $A \in \fA$. Let $\pi: \T_{\X}^{+} \rightarrow B(\H)$
be a completely contractive representation of $\T_{\X}^{+}$ and let $X \in B(\H)$ be a contraction satisfying,
\[
\pi(L)X = X \pi(\alpha(L)), \mbox{ for all } L \in \T_{\X}^{+}.
\]
Then there exist an isometric co-extension $\pi_1$ of
$\pi$ and an isometric co-extension $Z$ of $X$, so that
\[
\pi_1(L)Z=Z\pi_1(\alpha(L)), \mbox{ for all } L \in \T_{\X}^{+}.
\]
\end{thm}

\begin{proof} Notice that if $\pi_m$ is the minimal isometric dilation of $\pi$,
then $\pi_m\circ \alpha$ is the minimal isometric dilation of
$\pi \circ \alpha$.
Therefore, by applying commutant lifting to the covariance relations,
we obtain a contraction $X_1$ on a Hilbert space $\H_1$, satisfying
\[
\pi_m(L)X_1=X_1(\pi_m\circ \alpha)(L), \foral L \in \T^{+}_{\X}.
\]
Let $X_{1,m}$ be the minimal dilation of $X_1$, i.e.,
\begin{equation} \label{matricialform}
 X_{1,m}\simeq
  \begin{bmatrix}
   X_1 &  0 &  0 & 0 &  \dots \\
   D_{X_1}  &0 &  0  & 0 &  \dots \\
   0  &  I  & 0 &  0 & \dots \\
   0  &  0  & I & 0  &  \dots   \\
   \vdots & \vdots & \vdots & \vdots  & \ddots
   \end{bmatrix}
\end{equation}
where $D_{X_1} = (I-X_1^{*}X_1)^{1/2}$.
We apply now Remark \ref{rem;Solel}
to obtain completely contractive representations $\widehat{\pi}_m$ and $\widehat{\pi_m\circ\alpha}$, which
co-extend $\pi_m$ and $\pi_m\circ \alpha$, coincide on $\fA$, and satisfy
\begin{equation} \label{inter2}
 \widehat{\pi}_m(L) X_{1,m} = X_{1,m} \widehat{\pi_m \circ \alpha}(L) \qfor L \in \T^{+}_{\X}.
\end{equation}
Assume that these dilations have the form
\[
 \widehat{\pi}_m(L) =
 \begin{bmatrix} \pi_m(L) &  0  \\  Y^{(L)}& [Y^{(L)}_{jk}]_{j,k\ge1}  \end{bmatrix}
 \text{   and  \, }
 \widehat{\pi_m \circ \alpha}(L) =
 \begin{bmatrix} \pi_m \circ \alpha(L) &  0  \\ Z^{(L)}  & [Z^{(L)}_{jk}]_{j,k\ge1}  \end{bmatrix}
\]
with regards to the decomposition of the Hilbert space that corresponds to the matricial form of $X_{1,m}$ in (\ref{matricialform}).

\vspace{3mm}
{\bf{Claim:}} \textit{$Y^{(L)}=Z^{(L)}=0$, for all $L \in \T^{+}_{\X}$.}
\vspace{1mm}

Indeed, the claim is true in the case where $L = \pi_{\infty}(A)$, $ A \in \fA$, since the restrictions of $\widehat{\pi}_m$ and $\widehat{\pi_m \circ \alpha}$ on $\fA$ are $*$-homomorphisms dilating the $*$-homomorphisms $\pi_m$ and $\pi_m \circ \alpha$ respectively.
Hence it suffices to prove the claim in the case where $L = t_{\infty}(\xi)$, $\xi \in \X$. We show that $Y^{(\xi)} = 0$; a similar argument will show that $Z^{(t_{\infty}(\xi))} = 0$.
By the Schwarz inequality for completely contractive maps on unital operator algebras we have
\[
\widehat{\pi}_m\left(t_{\infty}(\xi)\right)^*\widehat{\pi}_m(t_{\infty}(\xi)) \leq \widehat{\pi}_m\big(\pi_{\infty}(\langle \, \xi \mid \xi \, \rangle)\big).
\]
By taking into account the matricial form of $\widehat{\pi}_m$ and comparing $(1,1)$-entries in the above inequality, we obtain
\begin{equation} \label{negative}
\pi_m\left(t_{\infty}(\xi)\right)^*\pi_m(t_{\infty}(\xi))+ (Y^{(t_{\infty}(\xi))})^* Y^{(t_{\infty}(\xi))} \leq \pi_m\big(\pi_{\infty}(\langle \, \xi \mid \xi \, \rangle)\big).
\end{equation}
However the map $\pi_m$ is an isometric representation and so
\[
\pi_m\left(t_{\infty}(\xi)\right)^*\pi_m(t_{\infty}(\xi))= \pi_m\big(\pi_{\infty}(\langle \, \xi \mid \xi \, \rangle)\big).
\]
and so (\ref{negative}) obtains
\[
(Y^{(t_{\infty}(\xi))})^* Y^{(t_{\infty}(\xi))} \leq 0,
\]
which proves the claim.
\vspace{3mm}

By comparing $(2,i)$-entries, $i=2,3,\dots$, in the covariance relation (\ref{inter2}) we also obtain
\[
Y_{2,i}^{(L)}=0, \quad \text{for all } i\geq2.
\]
In addition, by comparing $(i, 1)$-entries $i=3,4,\dots$, in (\ref{inter2}) we obtain
\[
Y_{i, 2}^{(L)}D_{X_1}=0, \quad \text{for all } i\geq3,
\]
and so $Y_{i, 2}^{(L)}=0, \text{for all } i\geq3$. This combined with the Claim implies that the second row and column of $\widehat{\pi}_m(L)$, $L \in \T^{+}_{\X}$, are equal to zero, except perhaps from $Y_{2, 2}^{(L)}$. Therefore, the map
\[
\rho: \T^{+}_{\G} \longrightarrow B(D_{X_1}(\H_1)),\,\, L\longmapsto Y_{2, 2}^{(L)}
\]
is a completely contractive representation of $\T^{+}_{\X}$.
By comparing $(2,1)$-entries in the covariance relation (\ref{inter2}), we now obtain
\begin{equation} \label{arewethereyet}
 \rho(L)D_{X_1} = D_{X_1} \pi_m \circ \alpha(L), \qfor L \in  \T^{+}_{\X}.
\end{equation}
For any $L \in \T^{+}_{\X}$, we now define
\[
\pi_m '(L)=
 \begin{bmatrix}
   \pi_m(L) &  0 &  0 & 0 &  \dots \\
   0  & \rho(L) &  0  & 0 &  \dots \\
   0  &  0  & \rho(\alpha(L)) &  0 & \dots \\
   0  &  0  &0     & \rho(\alpha^{(2)}(L))  &  \dots   \\
   \vdots & \vdots & \vdots & \vdots  & \ddots
   \end{bmatrix}
\]
By (\ref{arewethereyet}), this is a completely contractive representation $\pi_m '$
on a Hilbert space $\H_2$ so that
\[
   \pi_m '(L) X_{1,m} = X_{1,m} \pi_m '\circ \alpha(L) \qfor L \in \T^{+}_{\X}.
\]

Continuing in this fashion, we obtain a sequence
\[
     (\pi , X), \,(\pi_m, X_1),\, (\pi_m ', X_{1, m}), \,((\pi_m ')_m , X_2), \, ( ((\pi_m ')_m)', X_{2,m}) \dots
\]
of pairs of operators and representations acting on Hilbert spaces
$
  \H \subseteq \H_1 \subseteq \H_2 \dots,
$
co-extending $\pi$ and $X$ and satisfying the covariance relations.
Let $\H = \bigvee_j \H_j$, and consider these pairs as
acting on $\H$ by extending them to be zero on the complement.
Let
\[
 Z = \sotlim X_j = \sotlim X_{m, j}
\]
and define $\pi_1(L)$, $L \in \T^{+}_{\X}$, as a strong limit in a similar fashion.
These limits evidently exist as in both cases the sequences consist of either
isometries  or isometric representations that decompose as infinite direct sums.
Multiplication is \sot-continuous on the ball,
hence the covariance relations hold in the limit.
\end{proof}

Combining the Proposition above with Theorem \ref{main} we obtain the main result of the section,

\begin{cor} \label{maincor}
Let $(\X, \fA)$ be a $\ca$-correspondence, let $\alpha$ be a completely isometric
automorphism of the associated tensor algebra $\T_{\X}^{+}$ and assume that
$\alpha(A)=A$, for all $A \in \fA$. Then $\T_{\X}^{+} \times_{\alpha} \bbZ^+ $ and $\T_{\X}^{+} \times_{\alpha}^{\is} \bbZ^+$ are completely isometrically isomorphic and
\[
\cenv(\T_{\X}^{+} \times_{\alpha} \bbZ^+  )\simeq \O_{\X} \times_{\alpha} \bbZ.
\]
\end{cor}

In particular, the above corollary recaptures the main result of \cite{DavKatsdilation} with a different proof.

%%%%%%%%%%%%%%%%%%%%%%%%%%%%%%%%%%%%%%%%%%%%%


\begin{thebibliography}{99}

\bibitem{AlaP} M. Alaimia and J. Peters,
\textit{Semicrossed products generated by two commuting automorphisms}
J. Math. Anal. Appl. \textbf{285} (2003), 128--140.

\bibitem{Arv} W. Arveson,
\textit{Operator algebras and measure preserving automorphisms},
Acta Math. \textbf{118}, (1967), 95--109.

\bibitem{Arv3} W. Arveson,
\textit{Subalgebras of C*-algebras III},
Acta Math.\ \textbf{181} (1998), 159--228.

%\bibitem{Arv4} W. Arveson,
%\textit{The curvature invariant of a Hilbert module over $C[z\sb 1,\cdots,z\sb d]$},
%J. Reine Angew.\ Math.\ \textbf{522} (2000), 173--236.

%\bibitem{Arv5} W. Arveson,
%\textit{The Dirac operator of a commuting $d$-tuple},
%J. Funct.\ Anal.\ \textbf{189} (2002), 53--79.

\bibitem{ArvJ} W. Arveson and K. Josephson,
\textit{Operator algebras and measure preserving automorphisms II},
J. Functional Analysis \textbf{4}, (1969), 100--134.

\bibitem{BP} D. Buske and J. Peters,
\textit{Semicrossed products of the disk algebra: contractive representations
 and maximal ideals},
Pacific J. Math. \textbf{185} (1998), 97--113.

\bibitem{DeAPet} L. DeAlba and J. Peters,
\textit{Classification of semicrossed products of finite-dimensional C*-algebras},
Proc. Amer. Math. Soc. \textbf{95} (1985), 557--564.

\bibitem{DKconj} K. Davidson, E. Katsoulis,
\textit{Isomorphisms between topological conjugacy algebras},
J. reine angew.\ Math. \textbf{621} (2008), 29-51.

\bibitem{DKsurvey} K. Davidson, E. Katsoulis,
\textit{Nonself-adjoint crossed products and dynamical systems},
Contemporary Mathematics, to appear.

\bibitem{DavKatsdilation} K. Davidson, E. Katsoulis,
\textit{Dilating covariant representations of the non-commutative disc algebras},
J. Funct. Anal. \textbf{259} (2010), Pages 817-831.

\bibitem{DavKatsMem} K. Davidson, E. Katsoulis,
\textit{Operator algebras for multivariable dynamics}, Mem. Amer. Math. Soc., to appear.

\bibitem{DritsM} M. Dritschel and S. McCullough,
\textit{Boundary representations
for families of representations of operator algebras and spaces},
 J. Operator Theory \textbf{53} (2005), 159--167

\bibitem{Dru} S. Drury,
\textit{A generalization of von Neumann's inequality to the complex ball},
Proc.\ Amer.\ Math.\ Soc.\ \textbf{68} (1978), 300--304.

\bibitem{FR} N. Fowler, I. Raeburn,
\textit{The Toeplitz algebra of a Hilbert bimodule}, Indiana Univ.
Math. J. \textbf{48} (1999), 155--181.

\bibitem{HadH}  D. Hadwin and T. Hoover,
\textit{Operator algebras and the conjugacy of transformations.},
J. Funct. Anal. \textbf{77} (1988), 112--122.

\bibitem{Hoov} T. Hoover,
\textit{Isomorphic operator algebras and conjugate inner functions},
Michigan Math. J. \textbf{39} (1992), 229--237.

\bibitem{HPW} T. Hoover, J. Peters and W. Wogen,
\textit{Spectral properties of semicrossed products},
Houston J. Math. \textbf{19} (1993), 649--660.

\bibitem{Kakar} E. Kakariadis,
\textit{Semicrossed products and reflexivity},
J. Operator Theory, to appear.

\bibitem{Ka} T. Katsura,
\textit{On $\ca$-algebras
associated with $\ca$-correspondences},
J. Funct. Anal. \textbf{217} (2004), 366--401.

\bibitem{KaK} E. Katsoulis and D. Kribs,
\textit{Isomorphisms of algebras associated with directed graphs},
Math. Ann. \textbf{330}, (2004), 709--728.

\bibitem{KrP} D. Kribs and S. Power,
\textit{Free semigroupoid algebras},
J. Ramanujan Math. Soc., textbf{19} (2004), 75--117.

\bibitem{L} E. C. Lance,
\textit{Hilbert $C\sp *$-modules. A toolkit for operator algebraists},
London Mathematical
Society Lecture Note Series, \textbf{210}, Cambridge University Press, Cambridge, 1995.

\bibitem{MM} M. McAsey and P. Muhly,
\textit{Representations of nonselfadjoint crossed products},
Proc. London Math. Soc. (3) \textbf{47} (1983), 128--144.

\bibitem{MS2} P.S. Muhly, B. Solel,
\textit{Tensor algebras over $\ca$-correspondences:
representations, dilations and $\ca$-envelopes} J. Funct. Anal.
\textbf{158} (1998), 389--457.

\bibitem{MS} P. Muhly and B. Solel,
\textit{An algebraic characterization of boundary
representations},
Nonselfadjoint Operator Algebras, Operator Theory, and Related Topics,
Birkhäuser Verlag, Basel 1998, pp. 189--196.

\bibitem{Pet} J. Peters,
\textit{Semicrossed products of C*-algebras},
J. Funct. Anal. \textbf{59} (1984), 498--534.

\bibitem{Pet2} J. Peters,
\textit{The ideal structure of certain nonselfadjoint operator algebras},
Trans. Amer. Math. Soc. \textbf{305} (1988), 333--352.

\bibitem{Pop3} G. Popescu,
\textit{Non-commutative disc algebras and their representations}
Proc. Amer. Math. Soc. \textbf{124}, (1996), 2137--2148.

\bibitem{Var} N. Varopoulos,
\textit{On an inequality of von Neumann and an application of the metric theory of tensor products to operators theory},
J. Functional Analysis \textbf{16} (1974), 83--100.

\end{thebibliography}
\end{document}